\documentclass[11pt, reqno]{amsart}
\usepackage[margin=1.2in]{geometry}                
\usepackage{graphicx}
\usepackage{amssymb}
\usepackage{epstopdf}
\usepackage[colorlinks=true, pdfstartview=FitV, linkcolor=blue, citecolor=blue, urlcolor=blue]{hyperref}
 \usepackage{tikz}
\usepackage[listings]{tcolorbox}
\usepackage{caption}
\usepackage{subcaption}
\usepackage{mathrsfs}
\usepackage{enumitem}
\usepackage{verbatim}
\usepackage[smalltableaux]{ytableau}
\usepackage[nameinlink]{cleveref}
\crefname{defn}{definition}{definitions}
\Crefname{def}{Definition}{Definitions}
\crefname{thm}{theorem}{theorems}
\Crefname{thm}{Theorem}{Theorems}

\usepackage{color}

\newcommand{\N}{{\mathbb N}}
\renewcommand{\P}{{\mathbb P}}

\newcommand{\Q}{{\mathbb Q}}

\newcommand{\ba}{{\mathbf a}}
\newcommand{\bb}{{\mathbf b}}

\newcommand{\bx}{{\mathbf x}}

\newcommand{\kk}{{\sf k}}
\newcommand{\s}{\sigma}

\def\sym{S}

\def\S{\operatorname{Stab}}

\theoremstyle{plain} 
\newtheorem{thm}{Theorem}[section]

\newtheorem*{introthm*}{Theorem}

\newtheorem{cor}[thm]{Corollary}
\newtheorem{lem}[thm]{Lemma}
\newtheorem{prop}[thm]{Proposition}
\newtheorem{probl}[thm]{Problem}

\newtheorem*{oprobl*}{Open Problem}

\newtheorem{quest}[thm]{Question}

\theoremstyle{definition}
\newtheorem{defn}[thm]{Definition}

\newtheorem{ex}[thm]{Example}

\theoremstyle{remark}

\numberwithin{equation}{section}  

\title{Products and Powers of Principal Symmetric Ideals}

\author[Polymath 2023]{Eric Dannetun}
\address{LTH, Faculty of Engineering, Lund University, Lund, 22100, Sweden}
\author[]{Bruce Fang}
\address{Department of Mathematics and Statistics, Williams College, Williamstown, MA 01267 USA}
\author[]{Riccardo Formenti}
\address{Department of Mathematics, Universit\`{a} degli Studi di Milano, Milan, 20133, Italy}
\author[]{Bo Y. Gao}
\address{Department of Mathematics, University of Toronto, Toronto, Ontario, M5S 2E4, Canada}
\author[]{Juliann Geraci}
\address{Mathematics Department, University of Nebraska--Lincoln, Lincoln, NE, 68588, USA.}
\author[]{Ross Kogel}
\address{Department of Mathematics, Statistics, and Computer Science, Macalester College, Saint Paul, MN, 55105, USA.}
\author[]{Yuelin Li}
\address{Department of Mathematical Sciences,
Carnegie Mellon University, 
Pittsburgh, PA, 15213, USA.}
\author[]{Shreya Mandal}
\address{Ramakrishna Mission Vivekananda Educational and Research Institute, Belur, West Bengal, India.}
\author[]{Vinuge Rupasinghe}
\address{Department of Mathematics, University Of Colombo, Cumarathunga Munidasa Mawatha Colombo 00300, Sri Lanka.}
\author[]{Alexandra Seceleanu}
\address{Mathematics Department, University of Nebraska--Lincoln, Lincoln, NE, 68588, USA.}
\author[]{Duc Van Khanh Tran}
\address{Department of Mathematics, University of Texas at Austin, Austin, TX, 78712, USA.}
\author[]{Noah Walker}
\address{Department of Mathematics, Cedarville University, Cedarville, OH, 45314, USA.}


\begin{document}
\maketitle

\begin{abstract}
Principal symmetric ideals were recently introduced by Harada, Seceleanu, and \c{S}ega in \cite{HSS},  where their homological properties are elucidated. They are ideals generated by the orbit of a single polynomial under permutations of variables in a polynomial ring.  In this paper we determine when a product of two principal symmetric ideals is principal symmetric and when the powers of a principal symmetric ideal are again principal symmetric ideals.  We characterize the ideals that have the latter property as being generated by polynomials invariant up to a scalar multiple under permutation of variables.  Recognizing principal symmetric ideals is an open question for the purpose of which we produce certain obstructions. We also demonstrate that the Hilbert functions of  symmetric monomial ideals are not all given by symmetric monomial ideals, in contrast to the non-symmetric case.
\end{abstract}

\setcounter{tocdepth}{1} 
\tableofcontents

\section{Introduction}

Principal symmetric ideals, abbreviated in this paper as psi's, were introduced in \cite{HSS}. We recall this notion. Given an element $f$ of the polynomial ring $R=\kk[x_1,\ldots, x_d]$, the principal symmetric ideal $(f)_{S_d}$ is the ideal generated by the orbit of $f$ under the natural action of the symmetric group $S_d$ on $R$. In detail,  the definition states
\[
(f)_{S_d}=(f(x_{\sigma(1)}, \ldots, x_{\sigma(d)}) : \sigma\in S_d).
\]
Throughout this paper we make the assumption that the generator $f$ of a psi is a {\em homogeneous} polynomial. In this case the psi  $(f)_{S_d}$ is a homogeneous ideal and we may consider its invariants such as the minimal number of generators and the Hilbert function. 

In this paper we address a fundamental question.
\begin{quest}\label{q:recognition}
How can one determine whether a given ideal is a psi or not? 
\end{quest}

Whereas an ideal may be a psi despite not being written given in the form displayed above, this is a difficult question. An important ingredient in the remainder of our results consists of developing obstructions for an ideal to be a psi based on its number of generators and other characteristics. We investigate this paradigm  in more detail in \Cref{s:gens}. A highlight of our analysis  is  \Cref{negative order type analysis}, whose statement is too technical to state here.

Next we consider powers and products of psi's. Simple examples show that powers, and hence products of psi's need not be psi's; see \Cref{ex:k_sym_product}. We investigate the questions:

\begin{quest}\label{q:powers_products}
 When is a product of psi's a psi?
 When are the powers of a psi psi's?
\end{quest}

We settle the question regarding products completely in the case of monomial psi's and in the case of psi's in two dimensional polynomial rings in \Cref{product principal monomial} and \Cref{main product of psi}, respectively. This characterization is in terms of $\kk$-symetric polynomials, which are those only changed by a multiplication by a constant upon the action of any permutation of the variables. Regarding powers, we characterize psi's whose sufficiently large powers are also pis's as follows:

\begin{introthm*}[abstracted from \Cref{cor:powers monomial psi}, \Cref{thm:powers psi}, and \Cref{main power of psi}]
Let $I$ be a homogeneous principal symmetric ideal. Consider the statements
\begin{enumerate}
\item $I^n$ is a principal symmetric ideal for all $n\geq 1$
\item $I^n$ is a principal symmetric ideal for all sufficiently large integers $n$
\item $I^2$ is a principal symmetric ideal
\item $I$ is generated by a $\kk$-symmetric polynomial
\end{enumerate}
Statements (1), (2), and (4) are equivalent for all $I$. Statement (3) is equivalent to the rest when $I$ is a monomial ideal and when $I$ is an ideal of $\kk[x_1,x_2]$.
\end{introthm*}

F.S. Macaulay's classical work \cite{Macaulay} yields that for every homogeneous ideal $I$ there exists a monomial ideal having the same Hilbert function as $I$. We ask whether an analogous phenomenon takes place in the category of symmetric homogeneous ideals:

\begin{quest}\label{q:HF}
Given a symmetric ideal $I$, does there exist a symmetric monomial ideal with the same Hilbert function as $I$?
\end{quest}

In \Cref{thm:HF} we show that the answer to the question above is negative when $I$ ranges over a large class of psi's. Once again the key to this observation is an analysis of the number of generators of general psi's performed in \cite{HSS} versus that of monomial psi's in a large number of variables. We study monomial psi's in more detail in \Cref{s:monomial}, where we show the following useful characterization:

\begin{introthm*}[\Cref{principal monomial symmetric ideal}]
A principal symmetric ideal $I$  is a monomial ideal if and only if $I=(m)_{S_d}$ for some monomial $m$.
\end{introthm*}

Our work opens up avenues of investigation into a series of related questions which we list in \Cref{s:conjectures}.

\section{Terminology}

Throughout the paper $\N$ denotes the set of natural numbers (including zero) and $R$ denotes the polynomial ring $\kk[x_1,\ldots, x_d]$. All facts stated in this paper are true for $d=1$, but to avoid trivialities we assume henceforth that the number of variables is $d\geq 2$. The symmetric group $S_d$ acts on $R$ by means of setting for $\sigma\in S_d$ and $f\in R$
\[
\sigma\cdot f(x_1,\ldots, x_n)=f(x_{\sigma(1)},\ldots, x_{\sigma(n)}).
\]

Working with the symmetric group $S_d$ leads naturally to partitions with $d$ parts.

\begin{defn}
A {\em partition} is a tuple $\ba = (a_1, \cdots, a_d)\in\N^d$   such that $a_1 \geq \cdots \geq a_d$. Its number of parts is $d$ and the sum of its parts is $|\ba|=a_1+\cdots+a_d$. 

We denote by $P(n)$ the number of partitions of $n$ into arbitrary number of parts, that is,
\[
P(n)=\#\{\ba : |\ba|=n\}.
\]
\end{defn}

 Partitions are in bijection with equivalence classes of monomials under the action of the symmetric group.

\begin{defn}\label{def: Ra}
    Let $\ba = (a_1, \cdots, a_d)$ be a partition with $d$ parts. A monomial $m \in \kk[x_1, \cdots, x_d]$ is said to be of {\em order type} $\ba$ if 
    \[m = x_{1}^{a_{\sigma(1)}}\cdots x_{d}^{a_{\sigma(d)}} \text{ for some } \sigma \in S_d.
    \]
    We denote by $R_\ba$ the $\kk$-linear span of all monomials of order type $\ba$ in $R=\kk[x_1, \cdots, x_d]$ and by $(R_\ba)$ the ideal generated by this set.
\end{defn}

 Let $k$ denote the number of distinct values in the tuple $\ba$ and let $\{n_i\}_{i=1}^k$ be the number of occurrences of each value in $\ba$. Then we have:
    \begin{equation}\label{eq:dim Ra}
    \dim_\kk R_\ba = \frac{d!}{n_1!n_2!\cdots n_k!}.
    \end{equation}

We single out a class of polynomials which generate psi's having well-controlled generators. 

\begin{defn}
    Let $f$ be a homogeneous polynomial.
    We say that $f$ is {\em strongly homogeneous} if all monomials in $f$ have the same order type.
    If $f$ is strongly homogeneous, then the order type of $f$ is the order type of a monomial in $f$. 
\end{defn}
Note that the property of strong homogeneity and the order type of a polynomial are preserved under permutation of variables.

\begin{defn}\label{def:str homogeneous}
    An ideal $I$ is {\em strongly homogeneous} if $I$ is generated by strongly homogeneous polynomials $f_1,\ldots, f_n$ such that all $f_i$'s have the same order type.
      A {\em strongly homogeneous psi} is a principal symmetric ideal generated by a strongly homogeneous polynomial. 
\end{defn}

For example, principal symmetric ideals generated by a monomial are strongly homogeneous psi's. Note that a  strongly homogeneous psi is also a strongly homogeneous ideal in the sense of \Cref{def:str homogeneous}. Whether a homogeneous ideal is strongly homogeneous can be checked using any minimal generating set.

\begin{defn}
If $I$ is a homogeneous ideal of a graded $\kk$-algebra, it admits a decomposition $I=\bigoplus_{n\geq 0} I_n$, where $I_n$ is the $\kk$-vector space of elements of degree $n$ in $I$. 
The {\em Hilbert function} of $I$ is the function
$
 H_I:\N\to \N, \, H_I(n)=\dim_\kk I_n.
$
\end{defn}

\section{Recognizing principal symmetric ideals}\label{s:gens}

In this section we are concerned with the problem of identifying whether a given ideal is a psi. This is highly non trivial as the ideal may have a different set of generators than the one expected in the definition of a psi. Thus the difficulty of the problem consists in the non-uniqueness of minimal generating sets for homogeneous ideals.

In this section we focus on obstructions preventing a homogeneous ideal from being a psi. For a homogeneous ideal $I$ we denote by $\mu(I)$ its minimal number of generators. Since a psi $I$ has all its generators in a single degree, say degree $n$, one can compute the minimal number of generators of $I$ as $\mu(I)=\dim_\kk(I_n)$. The number of generators, when too large, can be construed as an obstruction to being a psi by means of the following observation.

\begin{lem}\label{min gen principal check}
    Let $I$ is homogeneous psi of $\kk[x_1,\ldots, x_d]$, then $\mu(I) \le d!$.
\end{lem}
\begin{proof}
By definition, a set of (not necessarily minimal) generators of $I$ is in bijection with the group $S_d$.
\end{proof}

For a polynomial $f$ we denote by $f({\bf 1})$ the element of $\kk$ obtained by evaluating $f$ at $x_i=1$ for $1\leq i\leq d$. Note that every element in the $S_d$ orbit of $f$ has the same value when evaluated at ${\bf 1}$. More can be said regarding obstructions to being a psi in terms of this invariant. 

\begin{lem}\label{dim of strongly homogeneous ideal}
Let $f$ be a strongly homogeneous polynomial in $\kk[x_1,\ldots, x_d]$ with order type $\ba$. If $f({\bf 1})= 0$ then $(f)_{S_d}\subsetneq (R_\ba)$. In particular $\mu ((f)_{S_d})\leq \dim_\kk R_\ba-1$.
\end{lem}
\begin{proof}
The containment $(f)_{S_d}\subsetneq (R_\ba)$ follows because $f\in R_\ba$ implies  $\sigma\cdot f\in R_\ba$ for any $\sigma\in S_d$. The condition $f({\bf 1})=0$ implies $(\sigma\cdot f)({\bf 1})=0$ for all $\sigma\in S_d$ and moreover $g({\bf 1})=0$ for any polynomial $g$ in the degree $|\ba|$ component of $I$. Therefore there are elements of $R_\ba$ not in $I$, for example, no monomials  of degree $\ba$ are in  $I$. This gives the strict containment and hence the claimed inequality.
\end{proof}

The following technical lemmas are the main ingredients in the proof of \Cref{negative order type analysis} as well as \Cref{principal symmetric monomial are strongly homogeneous}.

\begin{lem}\label{zero components}
    Let $f$ be a homogeneous polynomial and write $f=g_1+\cdots+g_r$, where the monomials in each of the $g_i$'s share no order types. Suppose there exists a homogeneous polynomial $g\in (f)_{S_d}$  such that $\deg(g)=\deg(f)$, $g$ shares order types only with $g_1$  among all $g_i$'s, and $g({\bf 1})\not=0$. Then for all $i\not=1$,  $g_i({\bf 1})=0$. Thus, $f({\bf 1})=g_1({\bf 1})$.
\end{lem}
\begin{proof}
    Suppose that $g\in (f)_{S_d}$ with $g({\bf 1})\not=0$.
    Since $\deg(g)=\deg(f)$, there exist constants $c_\sigma\in \kk$ such that 
    $$g=\sum_{\sigma\in S_d}c_\s\s\cdot(g_1+\cdots+g_r)=\sum_{\s\in S_d}c_\s\s\cdot g_1+\cdots+\sum_{\s\in S_d}c_\s\s\cdot g_r.$$
    Since $g$ shares order types only with $g_1$ and the $g_i$'s share no order types, we conclude from the identity displayed above that $g=\sum_{\s\in S_d}c_\s\s\cdot g_1$ and $\sum_{\s\in S_d}c_\s\s\cdot g_i=0$ for $i>1$.
    Hence,
    \begin{align*}
        g({\bf 1}) &= \sum_{\s\in S_d}c_\s(\s\cdot g_1)({\bf 1}) \\
                     &= \sum_{\s\in S_d}c_\s g_1({\bf 1}) \\
                     &= g_1({\bf 1})\sum_{\s\in S_d}c_\s.
    \end{align*}
    Since $g({\bf 1})\not=0$, it follows that $\sum\limits_{\s\in S_d}c_\s \neq 0$.
   For $i>1$ we deduce from $\sum_{\s\in S_d}c_\s\s\cdot g_i=0$ that
    \begin{align*}
        0 &= \sum_{\s\in S_d}c_\s(\s\cdot g_i)({\bf 1}) \\
          &= \sum_{\s\in S_d}c_\s g_i({\bf 1}) \\
          &= g_i({\bf 1})\sum_{\s\in S_d}c_\s 
         \end{align*}
         Since $\sum\limits_{\s\in S_d}c_\s \neq 0$, this yields $ g_i({\bf 1})=0$ for $i>1$, as desired, and thus $f({\bf 1})=g_1({\bf 1})$.
\end{proof}

\begin{lem}\label{order type split up}
    Let $f$ be homogeneous and $g\in (f)_{S_d}$ be a homogeneous polynomial such that $\deg(g)=\deg(f)$.
    Then $f=g'+h$ for some $g'$ which contains precisely the same order types as $g$ and for some $h$ which shares no order types with $g$.
    Also, if $g({\bf 1})\not=0$, then $g'({\bf 1})\not=0$.
\end{lem}
\begin{proof}
    Since $g\in (f)_{S_d}$ and $\deg(g)=\deg(f)$, there exist $c_\s\in K$ such that
    $$g=\sum_{\s\in S_d}c_\s\s\cdot f.$$
    Since order types are invariant under permutation, $f$ must contain a nonzero term for every order type present in $g$.
    Let $g'$ contain precisely those terms of $f$ whose order type is present in $g$.
    Let $h=f-g'$, so that $f=g'+h$ and $h$ shares no order types with $g$. 
  
    Now, suppose that $g({\bf 1})\not=0$.
    Note that
    $$g=\sum_{\s\in S_d}c_\s\s\cdot (g'+h)=\sum_{\s\in S_d}c_\s\s\cdot g'+\sum_{\s\in S_d}c_\s\s\cdot h.$$
    Since $h$ shares no order types with $g$, $\sum_{\s\in S_d}c_\s\s\cdot g'=g$ and $\sum_{\s\in S_d}c_\s\s\cdot h=0$.
    Thus,
    $$g({\bf 1})=\sum_{\s\in S_d}c_\s(\s\cdot g')({\bf 1})=g'({\bf 1})\sum_{\s\in S_d}c_\s,$$
    so $g'({\bf 1})\not=0$.
\end{proof}

We illustrate the previous lemma with an example.

\begin{ex}
Let $f=x_{1}^{3}-x_{2}^{3}+x_{1}^{2}x_{3}+x_{2}x_{3}x_{4}\in \Q[x_1,x_2,x_3,x_4]$. One can verify, for example by using Macaulay2 \cite{M2}, that $g=x_1x_2^2+x_1x_3x_4 \in (f)_{S_4}$. Among the terms of $f$, those which share order types with $g$ are $x_{1}^{2}x_{3}$ and $x_{2}x_{3}x_{4}$, while $x_{1}^{3}$ and $-x_{2}^{3}$ do not. So the proof of \Cref{order type split up} tells us to set $g'=x_{1}^{2}x_{3}+x_{2}x_{3}x_{4}$ and $h'=x_{1}^{3}-x_{2}^{3}$. Observe that $g({\bf 1})\not=0$ and also  $g'({\bf 1})\not=0$.
\end{ex}

\begin{thm}\label{negative order type analysis}
    Let $I$ be a homogeneous ideal generated in degree $n$.
    Let $g_1,g_2$ be homogeneous polynomials of degree $n$ that share no order types such that $g_1({\bf 1})\not=0$ and $g_2({\bf 1})\not=0$.
    If $g_1,g_2\in I$ then $I$ is not a principal symmetric ideal.
\end{thm}
\begin{proof}
    Assume to the contrary that $I$ is a principal symmetric ideal generated by some $f$.
    Then by \Cref{order type split up},  $f=g_1'+h_1$ and $f=g_2'+h_2$ where $g_i'$ and $h_i$ share no order types and moreover $g_i'$ contains precisely the same order types as $g_i$.
    Thus, the identity $g_1'+h_1=g_2'+h_2$ combined with the fact that $g_1, g_2$ share no order types shows that the terms of $g_1'$ are a subset of the terms of $h_2$ and the terms of $g_2'$ are a subset of the terms of $h_1$. 
   
       Let $h'=h_2-g_1'$ and note that $h'$ shares no order types with $g_1',g_2'$. Indeed, $h'$   shares no order types with $g_1'$ since all terms of $g_1'$ cancel in $h_2-g_1'$ and these are exactly the terms of $h_2$ of the same order type with terms in $g_1'$.   Moreover, $h'$   shares no order types with $g_2'$ since the terms of $h'$ are a subset of the terms of $h_2$ and $h_2$ shares no order types with $g_2'$.
       
        Then, $f=g_1'+g_2'+h'$. 
Since $g_1({\bf 1}),g_2({\bf 1})\not=0$, we may deduce $g_1'({\bf 1}),g_2'({\bf 1})\not=0$ by \Cref{order type split up}.
We now apply  \Cref{zero components} for  $g=g_1\in(f)_{S_d}$ based on the decomposition $f=g_1'+g_2'+h'$. To verify the hypothesis of this lemma, note that  $g_1$ shares order types only with $g_1'$, but not with $g_2'$ or $h'$. This is true since the order types of $g_2'$ are precisely the order types of $g_2$ and $g_1$ shares no order types with $g_2$, while $h'$ shares no order types with $g_1'$, as established above, but $g_1'$ has exactly the same order types as $g_1$. In view of this,   \Cref{zero components} yields that $g_2'({\bf 1})=0$ and $h'({\bf 1})=0$. The fomer identity provides the desired contradiction.    
\end{proof}

 \section{Monomial principal symmetric  ideals}\label{s:monomial}

As a matter of terminology, a {\em monomial  principal symmetric ideal} may be interpreted to be a psi that is also a monomial ideal or a psi generated by a monomial, that is, an ideal of the form $(m)_{S_d}$, where $m$ is a monomial. We prove in \Cref{principal monomial symmetric ideal} that these two potential interpretations describe the same class of ideals.

\begin{lem}\label{contains all monomials of an order type}
 Assume $(f)_{S_d}$ is a monomial ideal.
    If $f$ contains a term of order type $\ba$, then 
    $R_\ba\subseteq (f)_{S_d}$.
 If, in addition, $f$ is strongly homogeneous, then   $(R_\ba)=( f)_{S_d}$.
\end{lem}
\begin{proof}
    Suppose that $f$ contains a term which is a scalar multiple of a monomial $m$ of order type $\ba$ and let $I=(f)_{S_d}$. Since $f\in I$ and $I$ is assumed to be a monomial ideal, we deduce $m\in I$.  Since the ideal $I$ is symmetric, $(m)_{S_d}\subseteq (f)_{S_d}$. Since the monomials of order type $\ba$ are a basis for $R_\ba$ and any monomial of order type $\ba$ can be expressed as a permutation of $m$, we conclude that $(R_\ba)=(m)_{S_d}\subseteq (f)_{S_d}$.
    
  If  $f$ is strongly homogeneous order type $\ba$, then all terms of $f$ are in $R_\ba$, therefore $f\in R_\ba$ and $(f)_{S_d}\subseteq (R_\ba)$.
\end{proof}

\begin{lem}\label{strongly homogeneous components generate monomial ideals}
    Let $(f)_{S_d}$ be a monomial ideal, where $f$ is homogeneous. 
    Decompose $f$ into a sum of  strongly homogeneous components of pairwise distinct order types, $f=g_1+\cdots+g_t$.
    Then, for each $i$, $(g_i)_{S_d}$ is a monomial ideal and $(g_i)_{S_d}=(R_\ba)$ for some vector $\ba\in\N^d$ that depends on $i$.
\end{lem}
\begin{proof}
    Let $\ba$ be the order type of $g_i$.
    By \Cref{contains all monomials of an order type}, $$R_\ba\subseteq (f)_{S_d}=( g_1+\cdots+ g_t)_{S_d}\subseteq( g_1)_{S_d}+\cdots+(g_t)_{S_d}.$$
    Since $g_i$ has a different order type than the other $g_j$'s, $R_\ba\subseteq(g_i)_{S_d}$.
    Also, the minimal generators of $(g_i)_{S_d}$ form a subset of $R_\ba$, so $( g_i)_{S_d}=(R_\ba)$.    Thus $(g_i)_{S_d}$ is a monomial ideal.
\end{proof}


\begin{thm}\label{principal symmetric monomial are strongly homogeneous}
    Assume $(f)_{S_d}$ is a monomial ideal.
    If $f$ is homogeneous, then $f$ is strongly homogeneous and $(f)_{S_d}=(m)_{S_d}$ for some monomial $m$.
\end{thm}
\begin{proof}
    Assume to the contrary that $f$ contains monomial terms of different order types, $\ba_1,\dots,\ba_t$ say.
    Then we can write $f=g_1+\cdots+g_t$, where each $g_i$ is the sum of the terms of order type $\ba_i$.
    
    We show that 
    \begin{equation}\label{eq:monomialpsi}
    (g_1+\cdots+g_t)_{S_d}=(f)_{S_d}=(g_1)_{S_d}+\cdots+(g_t)_{S_d}.
    \end{equation}
    Clearly, one containment holds $$(f)_{S_d}=(g_1+\cdots+g_t)_{S_d}\subseteq(g_1)_{S_d}+\cdots+(g_t)_{S_d}.$$
    By  \Cref{strongly homogeneous components generate monomial ideals} we have $(g_i)_{S_d}=(R_{\ba_i})$ and by \Cref{contains all monomials of an order type} there are containments $R_{\ba_i}\subseteq(f)_{S_d}$.
    Hence, there is a containment
    $$(R_{\ba_i})+ \cdots+ (R_{\ba_t})=(g_1)_{S_d}+\cdots+(g_t)_{S_d}\subseteq(f)_{S_d}=(g_1+\cdots+g_t)_{S_d}.$$
    Since every term of $f$ is in $(R_{\ba_i})+ \cdots+ (R_{\ba_t})$, the claimed equality follows. In particular we see that $g_i\in (f)_{S_d}$ for each $i$. 
    
    Since $g_i$ is strongly homogeneous and $(g_i)_{S_d}=(R_{\ba_i})$, \Cref{dim of strongly homogeneous ideal} yields that $g_i({\bf 1})\not=0$ for all $1\leq i\leq t$. If $t>1$, then \Cref{negative order type analysis} contradicts the fact that $I$ is a psi. Hence $t=1$ and therefore $f$ is strongly homogeneous.

\end{proof}

\begin{cor}\label{principal monomial symmetric ideal}
    A principal symmetric ideal $I$  is a monomial ideal if and only if $I=(m)_{S_d}$ for some monomial $m$.
\end{cor}
\begin{proof}
The forward direction follows from \Cref{principal symmetric monomial are strongly homogeneous}. For the backward direction, note that all ideals  $(m)_{S_d}$, where $m$ is any monomial, are both psi's and monomial ideals.
   \end{proof}

\section{Products and powers of principal symmetric ideals}

Simple examples show that powers, and hence products,  of psi's need not be psi's.

\begin{ex}
Consider the ideal $(x_1)_{S_2}=(x_1,x_2)$. Its square $I^2=(x_1^2,x_1x_2,x_2^2)$ cannot be a psi in $\kk[x_1,x_2]$ since it has three minimal generators whereas a psi with respect to the action of the symmetric group $S_2$ has at most two generators by \Cref{min gen principal check}.
\end{ex}

A contrasting behavior is offered by the following family of examples.

\begin{ex}\label{ex:k_sym_product}
Suppose that $f$ is a $\kk$-symmetric polynomial and $g$ is arbitrary. Then $(f)_{S_d}(g)_{S_d}=(fg)_{S_d}$ and $\left((f)_{S_d}\right)^n=(f^n)_{S_d}$ are principal symmetric ideals
\end{ex}

In view of \Cref{ex:k_sym_product} one may wonder whether this class of examples characterizes the affirmative answer to \Cref{q:powers_products}. In \Cref{product principal monomial}, \Cref{cor:powers monomial psi}, \Cref {thm:powers psi}, and \Cref{main product of psi} we confirm that this is indeed the case. 

We start with a result on the number of generators of a product of psi's of importance to the future developments in this section.

\begin{thm}\label{min num gens of product 1}
    Let $I=(f)_{S_d}$ be a principal symmetric ideal and let $J=(g_1,\dots,g_m)$ be a homogeneous ideal generated in a single degree. Assume that $f$ is not $\kk$-symmetric.
    Then 
    \[\mu(IJ)\geq\mu(J)+1.
    \] 
      In particular, if $I,J$ are principal symmetric ideals neither of which is generated by a $\kk$-symmetric polynomial, then $\mu(IJ)\geq \max\{\mu(I),\mu(J)\}+1$.
\end{thm}
\begin{proof}
We may assume that $f$ has no $\kk$-symmetric divisor and that the polynomials $g_1,\ldots, g_m$ have no non-constant common divisor since dividing  by such divisors does not change any of the quantities involved. Then there is an irreducible factor of $f$, call it $p$, that does not divide all polynomials $\sigma\cdot f$. Indeed, if $p\mid \sigma\cdot f$ for all $\sigma\in S_d$ then $\sigma^{-1}\cdot p\mid f$ and therefore $\left(\prod_{\sigma\in S_d} \sigma\cdot p\right )\mid f$. Since the product  $\prod_{\sigma\in S_d} \sigma\cdot p$ is $\kk$-invariant we have obtained a contradiction.

 Pick $\sigma\in S_d$ so that $p\nmid \sigma\cdot f$. Let $U$ be the span of the minimal generators of $IJ$ and let $V$ be the span of the minimal generators of $J$. Let $h=\sigma\cdot f$. Then $U$ contains $fV+hV$ and consequently we have
  \begin{align*}
        \mu(IJ) &= \dim_\kk U \\
                &\geq  \dim_\kk( fV+hV) \\
                &\geq \dim_\kk (fV) +\dim_\kk (hV)-\dim_\kk(fV\cap hV) \\
                &\geq \dim_\kk (fV)+1 = \dim_\kk (V)+1 \\
                &= \mu(J)+1.
    \end{align*}
    The only inequality that requires explanation is $\dim_\kk (hV)-\dim_\kk(fV\cap hV)\geq 1$, equivalently $\dim_\kk (hV) > \dim_\kk(fV\cap hV)$ or $hV \supsetneq fV\cap hV$. Assume towards a contradiction that $hV = fV\cap hV$. Then each element of $hV$ is in $fV$. For example, $hg_i\in fV$ for each $i$, which yields that $f$, and therefore also $p$, divides $hg_i$ for each $i$. However, $p$ is coprime to $h$ by assumption, so $p$ divides each $g_i$. This contradicts the assumption that $g_1,\ldots, g_m$ have no non-constant common divisor, thus proving the claim.
\end{proof}

\subsection{Products and powers of monomial  principal symmetric ideals}

The following theorem answers \Cref{q:powers_products} for products of monomial psi's.

\begin{thm} \label{product principal monomial}
    Let $I,J$ be monomial  principal symmetric ideals.
    Then $IJ$ is a principal symmetric ideal if and only if $I=((x_1\cdots x_d)^n)_{S_d}$ or $J=((x_1\cdots x_d)^n)_{S_d}$ for some $n\geq0$.
\end{thm}
\begin{proof}
    $\Leftarrow:$ follows from \Cref{ex:k_sym_product}.
    \\
    $\Rightarrow:$
    Suppose that $I\not=((x_1\cdots x_d)^n)_{S_d}$ and $J\not=((x_1\cdots x_d)^n)_{S_d}$ for all $n\geq0$.
    Let $\ba,\bb$ be partitions such that $I=(\bx^\ba)_{S_d}$ and $J=(\bx^\bb)_{S_d}$, where $\bx^\ba=x_1^{a_1}\cdots x_d^{a_d}$ and $\bx^\bb=x_1^{b_1}\cdots x_d^{b_d}$.
    We now show that $IJ$ is not strongly homogeneous and thus not a monomial psi in view of \Cref{principal symmetric monomial are strongly homogeneous}.
    By assumption, $\ba,\bb\not=(n,\dots,n)$ for any $n$.
    Therefore, there exist $i,j$, which we pick to be smallest possible, such that $a_1\not=a_i$ and $b_1\not=b_j$.
    Without loss of generality assume that $j\geq i$.
       We can write 
    \begin{eqnarray*}
    \ba &=& (\underbrace{a_1,\dots,a_1}_{\mbox{\small $i-1$ times}},a_i,\dots)\\
     \bb &=& (\underbrace{b_1,\dots,b_1}_{\mbox{\small $j-1$ times}}, b_j,\dots)\\
      (1 \, j)\cdot\bb &=& (b_j,\underbrace{b_1,\dots,b_1}_{\mbox{\small $j-1$ times}},\dots).
\end{eqnarray*}
    Furthermore,
    $$\ba+\bb=(\underbrace{a_1+b_1,\dots,a_1+b_1}_{\mbox{\small $i-1$ times}}, a_i+b_1,\dots)$$
    and 
    $$\ba+(1 \, j)\cdot\bb=(a_1+b_j,\underbrace{a_1+b_1,\dots,a_1+b_1}_{\mbox{\small $i-2$ times}},a_i+b_1,\dots).$$

    Since $\ba,\bb$ are monotonic non-increasing vectors, $a_1+b_1\geq a_{i'}+b_{j'}$ for any $i',j'$.
    Furthermore, since $a_1>a_i\geq a_{i'}$ and $b_1>b_j\geq b_{j'}$ for $i'\geq i, j'\geq j$, every component of $\ba+(1 \, j)\cdot\bb$ after the first $i$ terms is strictly less than $a_1+b_1$.
    Therefore, $\ba+(1 \, j)\cdot\bb$ contains precisely $i-2$ components with maximum value, while $\ba+\bb$ contains $i-1$ components with maximum value.
    Hence, $\ba+\bb$ and $\ba+(1 \, j)\cdot\bb$ must have different order types.
    Therefore, since both $\bx^{\ba+\bb}$ and $\bx^{\ba+(1\, j)\cdot\bb}$ are minimal generators for $IJ$, $IJ$ is not strongly homogeneous and thus not principal monomial symmetric.
\end{proof}

\begin{cor}\label{cor:powers monomial psi}
Let $I$ be a monomial principal symmetric ideal of $\kk[x_1,\ldots, x_d]$. The following statements are equivalent:
\begin{enumerate}
\item $I^n$ is a principal symmetric ideal for all integers $n\geq 1$
\item $I^2$ is a principal symmetric ideal
\item $I$ is generated by a $\kk$-invariant monomial of the form $(x_1\cdots x_d)^t$ for some $t\geq 1$.
\end{enumerate}
\end{cor}
\begin{proof}
The implications $(1)\Rightarrow (2)$ and $(3)\Rightarrow (1)$ are clear. The implication $(2)\Rightarrow (3)$ follows from \Cref{product principal monomial} by setting $J=I$.
\end{proof}

Another way to phrase \Cref{product principal monomial} is that for monomial psi's $I$ and $J$, $IJ$ is a psi if and only if $\mu(I)=1$ or $\mu(J)=1$. \Cref{ex:order types product} shows that in general, for (not necessarily monomial) psi's $I,J$, the invariants $\mu(I),\mu(J)$ are not sufficient to determine if $IJ$ is principal symmetric.

\begin{ex}\label{ex:order types product}
Let $I=(x_1x_2)_{S_4}$ and $J=(x_1x_2+x_3x_4)_{S_4}$ and $J'=(x_1-x_2)_{S_4}$.
Then $\mu(J)=\mu(J')=3$ as computed using Macaulay2 \cite{M2}.
Note that $IJ$ is not principal symmetric by \Cref{negative order type analysis} since $x_1x_2x_3x_4+x_3^2x_4^2$ and $x_2x_3^2x_4+x_1x_3x_4^2$ are elements of minimal degree of $IJ$ that share no order types and are nonzero when evaluated at ${\bf 1}$.
However, $IJ'=(x_1x_2(x_1-x_3))_{S_4}$ is principal symmetric.
\end{ex}

\subsection{Powers of  principal symmetric ideals}

We now turn our attention to powers of arbitrary  pis's.  

%
%

\begin{thm}\label{thm:powers psi}
Let $I$ be a nonzero principal symmetric ideal. The following are equivalent:
\begin{enumerate}
\item $I^n$ is a principal symmetric ideal for all integers $n\geq 1$
\item $I^n$ is a principal symmetric ideal for all sufficiently large integers $n$
\item $I$ is generated by a $\kk$-invariant polynomial.
\end{enumerate}
\end{thm}
\begin{proof}
The implications $(1)\Rightarrow (2)$ and $(3)\Rightarrow (1)$ are clear. For $(2)\Rightarrow (3)$ assume towards a contradiction that $I$ is not generated by a $\kk$-invariant polynomial and note that applying \Cref{min num gens of product 1} inductively it follows that $\mu(I^n)\geq \mu(I) +n-1$. For sufficiently large $n$ we conclude that $\mu(I^n)>d!$, which contradicts \Cref{min gen principal check}.
\end{proof}

Comparing \Cref{cor:powers monomial psi} and  \Cref{thm:powers psi} we see that the result concerning monomial ideals is stronger. However, outside the case of monomial psi's it is not the case that $I^2$ being a psi forces $I$ to be generated by a $\kk$-invariant polynomial, as the following example demonstrates.

\begin{ex}\label{ex:d>2}
Let $\kk$ be a field of characteristic not equal to 2 and consider $I=(x_1-x_2)_{S_d}$. Note that $x_1-x_2$ is {\em not} a $\kk$-invariant polynomial in $\kk[x_1,\ldots,x_d]$ for $d>2$. We show that $I^2=((x_1-x_2)^2)_{S_d}$ is a psi. Since the inclusion $((x_1-x_2)^2)_{S_d}\subseteq I^2$ is evident, equality follows from expressing the generators of $I^2$ as elements of $((x_1-x_2)^2)_{S_d}$ as shown below
\begin{flalign*}
	& 2(x_{i}-x_{j})(x_{p}-x_{q}) \\
	 &=   (2x_{q}-x_{i}-x_{j})(x_{j}-x_{i}) + (-2x_{p}+x_{i}+x_{j})(x_{j}-x_{i}) \\
    &= (x_{q}-x_{i})^{2}-(x_{q}-x_{j})^{2}+(x_{j}-x_{p})^{2}-(x_{i}-x_{p})^{2}.
\end{flalign*}
\end{ex}

\subsection{The two variable case}

The case of psi's in the ring $\kk[x_1,x_2]$ closely resembles the monomial case.

\begin{thm} \label{main product of psi}
    Let $I=(f)_{S_2}$ and $J=(g)_{S_2}$ be homogeneous principal symmetric ideals of $K[x_1,x_2]$.  Then $IJ$ is a principal symmetric ideal if and only if $f$ or $g$ is $\kk$-symmetric.
\end{thm}
\begin{proof}
Assume that neither $f$ nor $g$ are $\kk$-symmetric. Since $f$ and $g$ are not $\kk$-symmetric $\mu(I),\mu(J)\not=1$, so $\mu(I)=2$ and $\mu(J)=2$ according to \Cref{min gen principal check}.  Therefore, by \Cref{min num gens of product 1}, it follows that $$\mu(IJ)\geq2+1=3>2=|S_2|.$$ Thus, by \Cref{min gen principal check}, $IJ$ is not a principal symmetric ideal.
\end{proof}

We have obtained an analogue of \Cref{cor:powers monomial psi}. 

\begin{cor}\label{main power of psi}
    Let $I=(f)_{S_2}$ be a homogeneous principal symmetric ideal of $\kk[x_1,x_2]$.
    Then $I^n$ is a principal symmetric ideal  for all integers $n\geq 1$ if and only if $I^2$ is a principal symmetric ideal if  and only if $f$ is $\kk$-symmetric.
\end{cor}

To see that the equivalence $I^2$ is a principal symmetric ideal if  and only if $f$ is $\kk$-symmetric does not extend to $d>2$ variables, consider \Cref{ex:d>2}.

\subsection{Analysis of products in terms of stabilizers}

\begin{defn}
  
    For $f\in \kk[x_1,\dots,x_d]$, let $\S_f=\{\sigma\in S_d : \sigma\cdot f=cf\mbox{ for some nonzero $c\in \kk$}\}$ be the $\kk$-stabilizer of $f$.
\end{defn}
Note that $\S_f$ is a subgroup of $S_d$. Also let $\iota$ be the identity permutation of $S_d$.
\begin{prop}\label{product of stabilizers}
    Let $I=(f)_{S_d}$ and $J=(g)_{S_d}$.
    If $S_d=\S_f\S_g$, then $IJ=(fg)_{S_d}$, so $IJ$ is a principal symmetric ideal.
\end{prop}
\begin{proof}
    By definition one has $IJ=((\sigma\cdot f)(\tau\cdot g) : \sigma,\tau\in S_d)$.
    For each $\sigma\cdot(fg)$ in the generating set of $(fg)_{S_d}$, $\sigma\cdot(fg)=(\sigma\cdot f)(\sigma\cdot g)$.
    Thus, $\sigma\cdot(fg)\in IJ$ yields $(fg)_{S_d}\subseteq IJ$.

    Given $(\sigma\cdot f)(\tau\cdot g)$, by assumption, $\sigma^{-1}\tau\in \S_f\S_g$, so $\sigma^{-1}\tau=\gamma^{-1}\delta^{-1}$ for some $\gamma\in \S_f$ and $\delta\in \S_g$.
    Hence, $\iota=\gamma\sigma^{-1}\tau\delta$.
    By definition of $\S_f$ and $\S_g$, $c_1(\sigma\gamma^{-1})\cdot f=\sigma\cdot f$ and $c_2(\tau\delta)\cdot g=\tau\cdot g$ for some nonzero $c_1,c_2\in \kk$.
    Therefore,
    $$\begin{aligned}
        (\sigma\cdot f)(\tau\cdot g) &= (c_1(\sigma\gamma^{-1})\cdot f)(c_2(\tau\delta)\cdot g) \\
        &= c_1c_2(\sigma\gamma^{-1}\cdot f)(\tau\delta\cdot g) \\
        &= c_1c_2(\sigma\gamma^{-1}\gamma\sigma^{-1})\cdot\left((\sigma\gamma^{-1}\cdot f)(\tau\delta\cdot g) \right)\\
        &= c_1c_2 (\sigma\gamma^{-1})\cdot \left((\gamma\sigma^{-1}\sigma\gamma^{-1}\cdot f)(\gamma\sigma^{-1}\tau\delta\cdot g) \right)\\
        &= c_1c_2(\sigma\gamma^{-1}\cdot(f(\gamma\sigma^{-1}\tau\delta\cdot g))) \\
        &= c_1c_2(\sigma\gamma^{-1}\cdot(fg)) \qquad\text{(since }\gamma\sigma^{-1}\tau\delta=\iota).
    \end{aligned}$$
The displayed computations show that $(\sigma\cdot f)(\tau\cdot g) \in (fg)_{S_d}$, thus, $IJ\subseteq (fg)_{S_d}$. 
\end{proof}
\begin{cor} \label{alt group gives principal}
    Let $I=(f)_{S_d}$ and $J=(g)_{S_d}$.
    If $\S_f=A_d$  (the alternating group) and $|\S_g|=2$, then $IJ=(fg)_{S_d}$ is a principal symmetric ideal.
\end{cor}
\begin{proof}
    Note that $\S_f\S_g=\{\sigma\tau : \sigma\in S_f,\tau\in S_g\}$.
    Furthermore, we show that all such products are distinct.
    Suppose that $\sigma\tau=\sigma'\tau'$ for some $\sigma,\sigma'\in \S_f$ and $\tau,\tau'\in \S_g$.
    If $\tau=\tau'$, then clearly, $\sigma=\sigma'$.
    If $\tau\not=\tau'$, then since $|\S_g|=2$, one of $\tau,\tau'$ is the identity $\iota.$
    Without loss of generality assume that $\tau=\iota$, then $\sigma=\sigma'\tau'$, so $\tau'=\sigma'^{-1}\sigma\in \S_f$.
    Now, $\tau'\in \S_g$ has order 2, so it is a transposition, which is an odd permutation.
    Hence $\tau'\not\in A_d=\S_f$, which is a contradiction.

    Since all the products in $\S_f\S_g$ are distinct, $|\S_f\S_g|=|\S_f||\S_g|=|S_d|$. Hence we have that $\S_f\S_g=S_d$, so by the above proposition, $IJ=(fg)_{S_d}$.
\end{proof}

\begin{ex}
Consider the polynomial ring $\kk[x_1,x_2,x_3]$ and let $f=x_1^2x_2+x_2^2x_3+x_3^2x_1$ and $g=x_1-x_2$. Then $\S_f=A_{3}, \S_g=\{(1\, 2),\iota\}$ and thus $(f)_{S_3}(g)_{S_3}=(fg)_{S_3}$ by \Cref{alt group gives principal}.
\end{ex}

\section{Hilbert functions of principal symmetric ideals}

A symmetric monomial ideal is a monomial ideal $J$ that is closed under the action of the symmetric group, that is, $\sigma\cdot m\in J$ for any $m\in J$ and $\sigma\in S_d$. In this section we consider \Cref{q:HF} which aims to compare the collection of Hilbert functions of psi's to that of symmetric monomial ideals (not necessarily monomial psi's). We are not aware of a characterization of all Hilbert functions attained by psi's. However, in \cite{HSS} the Hilbert function of ``most" psi's is determined. We now review the relevant facts.

Fix $n\in \mathbb N$. A parameter space for the set of principal symmetric ideals generated in degree $n$ is $\P^{N-1}$, where $N=\binom{n+d-1}{n}$.
Indeed, let $M_n$ denote the set of monomials of degree $n$ in $R=\kk[x_1,\ldots, x_d]$ listed as $M_n=\{m_1, \ldots, m_N\}$ in an arbitrary order.
Points in $\P^{N-1}$ parametrize principal symmetric ideals generated in degree $d$ via the assignment 
\begin{align*}
&\Phi: \P^{N-1} \to \text{ principal symmetric ideals} \\
& \Phi(c_1:\cdots: c_N) :=(f_c)_{\sym_n}\,,\qquad \text{where}\quad f_c=\sum_{i=1}^Nc_im_i\,.\qedhere
\end{align*}
The map $\Phi$ above is onto, but not one-to-one. It allows us to formulate a notion of a general principal symmetric ideal, which we now formalize. 
We will say that \emph{a general principal symmetric ideal generated in degree $d$ satisfies property $\mathcal P$} if there exists a non-empty Zariski-open set $U$ of $\P^N$ so that for each  $c\in U$, the principal symmetric ideal $\Phi(c)=(f_c)_{\sym_n}$ satisfies the property $\mathcal P$. The following was shown in \cite{HSS}:

\begin{thm}[Partial statement of {\cite[Theorem 8.4]{HSS}}]\label{HSS HF}
Suppose  ${\rm char}(\kk)=0$ and fix an integer $n\geq 2$. \footnote{Note that the meanings of $n$ and $d$ in this reference are interchanged compared to our notation here.} For sufficiently large  $d$, a general principal symmetric ideal $I$ of ${\kk}[x_1,\ldots, x_d]$ generated in degree $n$ yields a quotient ring  with Hilbert function
\[
 H_{R/I}(i):=\dim_\kk (R/I)_i = \begin{cases}\dim_\kk R_i &\text{if $i\le d-1$}\\
 P(n)-1&\text{if $i=d$}\\
 0 &\text{if $i>d$.}\
 \end{cases}
\]
\end{thm}

In the following we will refer to a psi $I$ or to a polynomial $f$ such that $I=(f)_{S_d}$ as general provided that it satisfies the conclusion of \Cref{HSS HF}. Heuristically, since non-empty Zariski open sets are dense, ``most" polynomials homogeneous $f$ of a fixed degree are general.

\begin{thm}\label{thm:HF}
    Let $\kk$ be a field of characteristic $0$, and let $f \in R = \kk[x_1,\ldots,x_d]$ be a polynomial of degree $n$. If $d$ is sufficiently large satisfying in particular $d \geq {n^n}/n!$ and $d>n$, and if $f$ is general then there exists no monomial symmetric ideal $J$ that satisfies $H_J(n) = H_{(f)_{S_d}}(n)$ for all $n\in \mathbb{N}$.
\end{thm}

\begin{proof}
   Let $J$ be a symmetric monomial ideal.    Using the notation in \Cref{def: Ra} we have a decomposition of the vector space $R_n$ of homogeneous polynomials of degree $n$ as 
       \begin{equation*}
        R_n = \bigoplus_{|\ba| = n}R_\ba,
    \end{equation*}
    where $\ba$ ranges over the set of partitions of sum $n$.
    Due to the fact that $J$ is symmetric, $R_\ba \cap J \neq 0$ implies $R_\ba \subset J$. Thus there is an induced  decomposition of $J_n$ of the form
    \begin{equation*}
        J_n = J \cap R_n = \bigoplus_{ \substack{|\ba| = n \\  R_\ba \cap J \neq 0}} R_\ba
    \end{equation*}
    and consequently we obtain
    \begin{equation}\label{eq: HF J}
    H_R(n)=  \sum_{ \substack{|\ba| = n }} \dim_k R_\ba \qquad \text{and} \qquad
       H_J(n)=  \sum_{ \substack{|\ba| = n \\  R_\ba \cap J \neq 0 }} \dim_k R_\ba.
    \end{equation}
    We write this concisely as
       \begin{equation}\label{eq:HJ}
    H_{R/J}(n)=   H_R(n)-H_J(n)=  \sum_{ \substack{|\ba| = n \\  R_\ba \cap J = 0 }} \dim_k R_\ba.
       \end{equation}
 Let $\ba$ be a partition of $n$, let $k$ denote the number of distinct values in $\ba$ and $\{n_i\}_{i=1}^k$ the number of occurrences of each value. By \Cref{eq:dim Ra} we have
    \begin{equation}\label{eq: 6.3}
     \dim_\kk R_\ba = \frac{d!}{n_1!n_2!\cdots n_k!}.
    \end{equation}
    A partition $\ba$ of $n$ has at most $n$ parts, where the possible value of the $i$-th part $a_i$ satisfies $0\leq a_i\leq  \frac{n}{i}$. Indeed, the inequality $ia_i\leq |a|=n$ can be deduced from the fact that the parts are ordered non-increasingly and the sum of the first $i$ parts does not exceed $|a|$. So we must have the following upper bound for $P(n)$
    \begin{equation*}
        P(n) \leq \prod_{i=1}^n \frac{n}{i} =\frac{n^n}{n!} \leq d.
    \end{equation*}

   Since $\ba$ is a partition of $n$ having $d$ parts and since $d>n$, no part can be repeated $d$ times as this would imply that $n=|a|$ is divisible by $d$, a contradiction.  Thus $n_i<d$ for each $1\leq i\leq k$ and in particular we have that $k\geq 2$. Since there are $d$ factors in the product  $n_1!n_2!\cdots n_k!$ at least two of which are equal to 1, we deduce that
   \begin{equation}\label{eq: contrad}
   n_1!n_2!\cdots n_k!\leq (d-1)! 1!, \quad \text{ i.e., } \quad \dim_\kk R_\ba= \frac{d!}{n_1!n_2!\cdots n_k!}\geq \frac{d!}{(d-1)!1!}=d \geq P(n).
\end{equation}
    
      Let $n = \deg f$ and  assume that there is some  symmetric monomial ideal $J$ such that $H_J(n) = H_{(f)_{S_d}}(n)$ for all $n \in \mathbb{N}$.
 By \Cref{HSS HF}, for sufficiently large  $d$ and  general $f$ we have
    \begin{equation}\label{eq:HSS}
   H_R(n )-  H_J(n)=   H_R(n )-  H_{(f)_{S_d}}(n) = P(n) -1.
    \end{equation}
\Cref{eq:HSS} and \Cref{eq:HJ} show there exists a partition $\ba$ of $n$ with $d$ parts such that
\[
\dim_\kk R_\ba\leq P(n)-1.
\]
 This contradicts \Cref{eq: contrad}, which yields $\dim_\kk R_\ba\geq P(n)$ for every partition $\ba$.
\end{proof}

Next we consider the case of two variables.

\begin{defn}
    For $f \in \kk[x_1,x_2]$ and $\sigma=(1 \, 2)$ let $g = \gcd(f,\sigma \cdot f)$. The {\em symmetric reduction} of $f$, or $\text{sRed}(f)$ is defined as $\text{sRed}(f) = f/g$.
\end{defn}
Note that the  symmetric reduction of a polynomial is only well defined up to multiplication by a unit. However its degree is well defined and this suffices for our purpose.

We observe that the Hilbert function of a psi in the two-dimensional ring depends only on the degree of the generator and its symmetric reduction.  In the following, for a graded module $M$ we employ the standard notation $M(-i)$ to mean the graded module which has the same underlying set as $M$ and satisfies $M(-i)_n=M_{n-i}$.

\begin{lem}\label{lem: HF 2dim}
  Let $f \in R=\kk[x_1,x_2]$ be homogeneous, let $k=\deg (f)$ and $\ell=\deg(\rm{sRed}(f))$.  Then for every integer $n$ the following holds \footnote{Recall that $H_R(i)=0$ for $i<0$.}
  \[
H_I(n)= 2H_R(n-k) - H_R(n-k-\ell).
\]
\end{lem}
\begin{proof}
Set  $I=(f)_{S_d}$ and $g = \gcd(f,\sigma \cdot f)$. Assume first that $f$ is not $\kk$-symmetric, therefore $\ell>0$ and $\deg(g)=k-\ell$. The Hilbert function of $I$ can be deduced from the isomorphism
\[
I=g\left(\text{sRed}(f), \sigma \cdot \text{sRed}(f)\right)\cong \left(\text{sRed}(f), \sigma \cdot \text{sRed}(f)\right)(-k+\ell),
\]
where $C=\left(\text{sRed}(f), \sigma \cdot \text{sRed}(f)\right)$ is generated by a regular sequence. The isomorphism provides the equality $H_I(n)=H_C(n-k+\ell)$ and the Hilbert function of $C$ is obtained from the Koszul complex
\[
0\to R(-2\ell)\to R^2(-\ell)\to C\to 0
\]
which yields  $H_C(n)=2H_R(n-\ell)-H_R(n-2\ell)$. This establishes the claimed formula in the case $\ell > 0$.

Assume now that $f$ is $\kk$-symmetric, therefore $\ell=0$. Then, as $(f)_{S_2}=(f)$ is a principal ideal in the usual sense we have $H_{(f)_{S_2}}(n) =H_R(n-k)$. This agrees with the claim.
\end{proof}

\begin{prop}\label{cor:prob8_2dim}
     Let $f \in \kk[x_1,x_2]$ be homogeneous, let $\ell=\deg(\text{\rm sRed}(f))$ and $k=\deg(f)$.  There exists a monomial symmetric ideal $J$, such that $H_{(f)_{S_2}}(n) = H_{J}(n)$ if and only if $k+\ell$ is even. 
     \end{prop}
\begin{proof}
Set $I=(f)_{S_2}$. 
From \Cref{lem: HF 2dim} applied in degree $n>k+\ell$ with $n$ odd we obtain 
\[
H_I(n)=2H_R(n-k) - (n-k-\ell+1)\equiv n-k-\ell+1 \pmod{2}.
\]
Suppose that $J$ is a symmetric monomial  ideal  such that $H_{(f)_{S_2}}(n) = H_{J}(n)$ for all $n\in \N$. Fix $n$ to be an odd integer. Then since $d=2$ and there are no partitions $\ba$ of an odd integer $n$ with two equal parts we see by \Cref{eq:dim Ra} that $\dim_\kk R_\ba=2$ for all $\ba$ with $|\ba|=n$. Consequently by \Cref{eq: HF J}, $H_J(n)$ is even. Since $H_J(n)=H_I(n)\equiv k+\ell \pmod{2}$ we conclude that $k+\ell$ must be even.

Conversely, assume $k+\ell$ is even, set $a = \frac{k-\ell}{2}$ and $g = x_1^{a+\ell}x_2^a$. By \Cref{lem: HF 2dim} the Hilbert function of $J =(g)_{S_2}$ is uniquely determined by $\deg(\text{sRed}(g))$ and  $\deg(g)$. Since $\text{sRed}(g) = x_1^\ell$, which has degree $\ell$ and since $\deg(g) = k = \deg(f)$ we obtain $H_{(g)_{S_2}}(n) = H_I(n)$.

\end{proof}

\section{Open problems}\label{s:conjectures}

Satisfactory partial answers to \Cref{q:recognition} and \Cref{q:powers_products} of the introduction have been provided in this manuscript, however these questions are still open in full generality.  We single out a few additional questions which result from our work.

In view of \Cref{ex:d>2}, where the polynomial $x_1-x_2$ symmetrically generating the relevant ideal is $\kk$-invariant under the action of the subgroup $S_2$ of $S_d$ we ask:
\begin{quest}
Is it true that if $I$ and $I^2$  are principal symmetric ideals, then $I=(f)_{S_d}$ for some polynomial $f$ which is $\kk$-invariant under the action of a symmetric subgroup of $S_d$?
\end{quest}

We remark that \cite[Lemma 5.1]{HSS} and \Cref{HSS HF} give an upper bound on the minimal number of generators of any principal symmetric ideal $I$ generated by a polynomial of degree $n$ in the ring  $\kk[x_1, \ldots, x_d]$ with $d$ sufficiently large, namely 
\[
\mu(I)\leq \binom{n+d-1}{d-1}-P(n)+1.
\]
On the other hand several of our results proceed by way of estimating {\em lower bounds} on the number of minimal generators of various ideals as well as their products and powers; see for example \Cref{min num gens of product 1}.  It is therefore of interest to continue the search for such bounds.

\begin{probl}
Find (nontrivial) upper and lower bounds for the number of minimal generators of the product of two homogeneous principal symmetric ideals. 

Find (nontrivial) upper and lower bounds for the number of minimal generators of the powers of a homogeneous principal symmetric ideal.
\end{probl}

Finally, \Cref{thm:HF} opens up the search for a non-monomial ideal analogue of Macaulay's theorem in the category of homogeneous symmetric ideals. Evidence from \cite{HSS} suggests that binomial ideals could replace monomial ideals in the desired result. 

\begin{quest}\label{q: binomial}
If $\kk$ is a field of characteristic zero and $I$ is a  homogeneous symmetric ideal of $\kk[x_1,  \ldots, x_n]$ that does not have  $\kk$-symmetric polynomials in any set of minimal generators, does there exists a symmetric ideal $J$, generated by monomials and binomials, so that $H_I (n) = H_J (n)$ for all $n\in \N$?
\end{quest}

The answer to \Cref{q: binomial} is affirmative if $I$ is a general psi by \cite[Proposition 6.12]{HSS} and \Cref{HSS HF}. The answer is also affirmative in two-dimensional polynomial rings  without any restriction on the generators of $I$, as we show next.

\begin{prop}\label{cor:2-dim_binom}
    Let $\kk$ be a field of characteristic different from two and let $f\in \kk[x_1,x_2]$ be a homogeneous polynomial. Then there exists a symmetric ideal $J$ generated by binomials or monomials such that $H_J(n) = H_{(f)_{S_2}}(n)$
\end{prop}
\begin{proof}
Let $\ell =\deg( \text{sRed}(f))$, $k = \deg f\geq \ell$. Consider the binomial 
\[ g = x_1^\ell x_2^{k-\ell} + x_1^k = x_1^\ell(x_2^{k-\ell} + x_1^{k-\ell}).
\] Thus $\deg(g)=k$ and since $\deg(\text{sRed}(g)) = \ell$, \Cref{lem: HF 2dim} gives that $H_{(f)_{S_2}}(n) = H_{(g)_{S_2}}(n)$.
\end{proof}


\paragraph{\bf Acknowledgements.} We thank Matthew Hase-Liu for comments which improved this manuscript. We thank the referee for pointing out a flaw in the proof of \Cref{negative order type analysis}, which we subsequently remedied.  Computations leading to several results in this paper were performed with Macaulay2 \cite{M2}. We acknowledge the support of NSF DMS--2218374 for the 2023 Polymath Jr. program, which facilitated our collaboration. Geraci and Seceleanu were also partially supported by NSF DMS--2101225.


\end{document}